\documentclass[a4paper,12pt]{article}
\usepackage{a4wide}
\usepackage{amsmath}
\usepackage{amssymb}
\usepackage{amsthm}
\usepackage{latexsym}
\usepackage{graphicx}
\usepackage[english]{babel}
\usepackage{makeidx}

\newtheorem{obs} [subsection]{Remark}
\newtheorem{exm} [subsection]{Example}

\newtheorem{prop}[subsection]{Proposition}

\newtheorem{teor}[subsection]{Theorem}
\newtheorem{lema}[subsection]{Lemma}
\newtheorem{cor} [subsection]{Corollary}
\newcommand{\Zng}{$\mathbb Z^n$-graded $S$-module}

\def\sdepth{\operatorname{sdepth}}
\def\depth{\operatorname{depth}}

\def\deg{\operatorname{deg}}
\def\lcm{\operatorname{lcm}}

\def\reg{\operatorname{reg}}
\begin{document}
\selectlanguage{english}
\frenchspacing

\begin{center}
\textbf{A class of square-free monomial ideals associated to two integer sequences}
\vspace{10pt}

\large
Mircea Cimpoea\c s
\end{center}
\normalsize

\begin{abstract}
Given two finite sequences of positive integers $\alpha$ and $\beta$, we associate a square free monomial ideal $I_{\alpha,\beta}$ in a ring of polynomials $S$, and we recursively compute the algebraic invariants of $S/I_{\alpha,\beta}$. Also, we give precise formulas in special cases.


\end{abstract}

\section*{Introduction}

Let $K$ be a field and $S=K[x_1,\ldots,x_n]$ the polynomial ring over $K$.
Let $M$ be a \Zng. A \emph{Stanley decomposition} of $M$ is a direct sum $\mathcal D: M = \bigoplus_{i=1}^rm_i K[Z_i]$ as a $\mathbb Z^n$-graded $K$-vector space, where $m_i\in M$ is homogeneous with respect to $\mathbb Z^n$-grading, $Z_i\subset\{x_1,\ldots,x_n\}$ such that $m_i K[Z_i] = \{um_i:\; u\in K[Z_i] \}\subset M$ is a free $K[Z_i]$-submodule of $M$. We define $\sdepth(\mathcal D)=\min_{i=1,\ldots,r} |Z_i|$ and $\sdepth(M)=\max\{\sdepth(\mathcal D)|\;\mathcal D$ is a Stanley decomposition of $M\}$. The number $\sdepth(M)$ is called the \emph{Stanley depth} of $M$. 

Herzog, Vladoiu and Zheng show in \cite{hvz} that $\sdepth(M)$ can be computed in a finite number of steps if $M=I/J$, where $J\subset I\subset S$ are monomial ideals. In \cite{rin}, Rinaldo give a computer implementation for this algorithm, in the computer algebra system $\mathtt{CoCoA}$ \cite{cocoa}. In \cite{apel}, J.\ Apel restated a conjecture firstly given by Stanley in \cite{stan}, namely that $\sdepth(M)\geq\depth(M)$ for any \Zng $\;M$. This conjecture proves to be false, in general, for $M=S/I$ and $M=J/I$, where $0\neq I\subset J\subset S$ are monomial ideals, see \cite{duval}. For a friendly introduction in the thematic of Stanley depth, we refer the reader \cite{her}.

Stanley depth is an important combinatorial invariant and deserves a thorough study. The explicit computation of the Stanley depth is a
difficult task. Also, although the Stanley conjecture was disproved, it is interesting to find large classes of ideals which satisfy the Stanley inequality. In our paper, we consider certain classes of square free monomial ideals for which we can apply inductive methods in order to study the Stanley depth invariant. We consider that our approach could be useful to study other classes of ideals.

Given two sequences of positive integers 
$$\alpha: a_1<a_2<\cdots<a_s,\; \beta: b_1<b_2<\cdots<b_s,$$ with $a_i<b_i$ for all $1\leq i \leq s$ and $b_s\leq n$, we consider the square-free monomial ideal $$I_{\alpha,\beta}=(x_{a_1}\cdots x_{b_1},\ldots,x_{a_s}\cdots x_{b_s})\subset S=K[x_1,\ldots,x_n].$$ 
Note that $I_{\alpha,\beta}$ is a natural generalization for the path ideal associated to the path graph, see \cite{path} for further details. The main goal of our paper is to study the algebraic and combinatorial invariants of the ideal $I_{\alpha,\beta}$.

\footnotetext[1]{The support from grant ID-PCE-2011-1023 of Romanian Ministry of Education, Research and Innovation is gratefully acknowledged.}

\newpage
For $s\geq 1$, let $j:=j(\alpha,\beta):=\max\{i:\; a_i\leq b_1\}$. We consider the sequences $\alpha': a'_1=a_{j+1} < a'_2=a_{j+2}< \cdots < a'_{s-j}=a_s$ and $\beta': b'_1=b_{j+1} < b'_2=b_{j+2} < \cdots < b'_{s-j}=b_s$. Assume $s>1$ and $j>1$. If $a_{j+1}>b_1+1$, we define $\alpha'': a''_1=b_1+1 < a''_2=a_{j+1} < \cdots < a''_{s-j+1}=a_s$ and
$\beta'': b''_1=b_2 < b''_2 = b_{j+1} < \cdots < b''_{s-j+1}=b_s$. If $a_{j+1}=b_1+1$, we define $\alpha'': a''_1=b_1+1 < a''_2=a_{j+2}< \cdots < a''_{s-j}=a_s$ and $\beta'': b''_1=b_2 < b''_2 = b_{j+2} < \cdots < b''_{s-j}=b_s$.
We define, recursively, the following numbers:
$$ \varphi(\alpha,\beta):= \begin{cases} n,\;s=0 \\ \varphi(\alpha',\beta')-1,\;s\geq 1, j=1 \\ \varphi(\alpha'',\beta'')-1,\;s>1 ,j > 1
 \end{cases},\; \psi(\alpha,\beta):= \begin{cases} n,\;s=0 \\ \psi(\alpha',\beta')-1,\;s\geq 1 \end{cases}.$$
In Theorem $1.6$, we prove that
$$\sdepth(S/I_{\alpha,\beta}) =\depth(S/I_{\alpha,\beta}) =\varphi(\alpha,\beta),\; \dim(S/I_{\alpha,\beta})=\psi(\alpha,\beta).$$
For special cases of $\alpha$ and $\beta$, we give precise formulas for $\sdepth(S/I_{\alpha,\beta})$ and $\depth(S/I_{\alpha,\beta})$, see Proposition $1.9$, Theorem $1.10$ and Proposition $1.11$. Also, we prove that $$ \sdepth(I_{\alpha,\beta})\geq \depth(I_{\alpha,\beta}),$$ see Proposition $1.7$. For $n\geq m\geq 1$, the \emph{$m$-path ideal of the path graph} of length $n$ is
$$I_{n,m}=(x_1x_2\cdots x_m, x_2x_3\cdots x_{m+1},\ldots,x_{n-m+1}\cdots x_n)\subset S.$$
As a consequence of Theorem $1.6$, we give another prove of \cite[Theorem 1.3]{path}, where we computed $\sdepth(S/I_{n,m})$.

In the second section, we give algorithms to compute the Hilbert series and the Betti numbers of the ideals $I_{\alpha,\beta}$, see Proposition $2.1$ and Proposition $2.3$, using the Eliahou-Kervaire resolution \cite{eli}.

\section{Main results}

First, we recall the well known Depth Lemma, which can be found in any standard text of commutative algebra.
We present the statement given in \cite[Lemma 1.3.9]{real}. 

\begin{lema}(Depth Lemma)
If $0 \rightarrow U \rightarrow M \rightarrow N \rightarrow 0$ is a short exact sequence of modules over a local ring $S$, or a Noetherian graded ring with $S_0$ local, then:

$(1)$ $\depth M \geq \min\{\depth N,\depth U\}$.

$(2)$ $\depth U \geq \min\{\depth M,\depth N +1\}$.

$(3)$ $\depth N \geq \min\{\depth U-1,\depth M\}$.
\end{lema}

In \cite{asia}, Asia Rauf proved the analog of Lemma $1.1(1)$ for $\sdepth$:

\begin{lema}
Let $0 \rightarrow U \rightarrow M \rightarrow N \rightarrow 0$ be a short exact sequence of $\mathbb Z^n$-graded $S$-modules. Then:
$ \sdepth(M) \geq \min\{\sdepth(U),\sdepth(N) \}$. 
\end{lema}

We also recall the following well known results. See for instance \cite[Corollary 1.3]{asia}, \cite[Proposition 2.7]{mirci} and \cite[Corollary 3.3]{asia}.

\begin{lema}
Let $I\subset S$ be a monomial ideal and let $u\in S$ a monomial which is not in $I$. 

$(1)$ $\sdepth(S/(I:u))\geq \sdepth(S/I)$, $\sdepth(I:u)\geq \sdepth(I)$, $\depth(S/(I:u))\geq \depth(S/I)$.

$(2)$ If $u$ is regular on $S/I$, then $\sdepth(S/(I,u))=\sdepth(S/I)-1$.
\end{lema}

Also, we recall the following result from \cite{asia}.

\begin{lema}(\cite[Theorem 3.1]{asia})
Let $I\subset S_1=K[x_1,\ldots,x_m]$, $J\subset S_2=K[x_{m+1},\ldots,x_n]$ be two monomial ideals and let 
$S=S_1\otimes_K S_2=K[x_1,\ldots,x_n]$. It holds that
\[ \sdepth(S/(IS,JS))\geq \sdepth_{S_1}(S_1/I) + \sdepth_{S_2}(S_2/J).\]
\end{lema}

Let $0\leq s \leq n$ be an integer.
We consider two sequences of integers $\alpha: a_1 < a_2 <\cdots <a_s$ and $\beta: b_1 < b_2 <\cdots <b_s$ with 
$1\leq a_1$, $b_s\leq n$ and $a_i\leq b_i$, for all $1\leq i \leq s$. If $s=0$, then $\alpha$ and $\beta$ are the empty set. 

For $s\geq 1$, let $j:=j(\alpha,\beta):=\max\{i:\; a_i\leq b_1\}$. We consider the sequences $\alpha': a'_1=a_{j+1} < a'_2=a_{j+2}< \cdots < a'_{s-j}=a_s$ and $\beta': b'_1=b_{j+1} < b'_2=b_{j+2} < \cdots < b'_{s-j}=b_s$, obtained as restrictions of the initial sequences 
$\alpha$ and $\beta$.

Assume $s>1$ and $j>1$. If $a_{j+1}>b_1+1$, we define $\alpha'': a''_1=b_1+1 < a''_2=a_{j+1} < \cdots < a''_{s-j+1}=a_s$ and
$\beta'': b''_1=b_2 < b''_2 = b_{j+1} < \cdots < b''_{s-j+1}=b_s$. If $a_{j+1}=b_1+1$, we define $\alpha'': a''_1=b_1+1 < a''_2=a_{j+2}< \cdots < a''_{s-j}=a_s$ and $\beta'': b''_1=b_2 < b''_2 = b_{j+2} < \cdots < b''_{s-j}=b_s$.

We define, using the new sequences $\alpha',\beta',\alpha''$ and $\beta''$, recursively, the following numbers
$$ \varphi(\alpha,\beta):= \begin{cases} n,\;s=0 \\ \varphi(\alpha',\beta')-1,\;s\geq 1, j=1 \\ \varphi(\alpha'',\beta'')-1,\;s>1 ,j > 1
 \end{cases},\; \psi(\alpha,\beta):= \begin{cases} n,\;s=0 \\ \psi(\alpha',\beta')-1,\;s\geq 1 \end{cases}.$$

\begin{obs}
\emph{If $s=1$, then $\varphi(\alpha,\beta)=\psi(\alpha,\beta) = n-1$, since $\alpha'$ and $\beta'$ are the empty sets.}

\emph{If $s=2$, then $\varphi(\alpha,\beta)=n-2$. Indeed, if $b_1<a_2$, i.e. $j(\alpha,\beta)=1$, then $\varphi(\alpha,\beta)=\varphi(\alpha',\beta')-1=(n-1)-1$. If $b_1\geq a_2$, then $I_{\alpha'',\beta''} = (x_{b_1+1}\cdots x_{b_2})$ and thus $\varphi(\alpha,\beta)=\varphi(\alpha'',\beta'')-1=n-2$.}
\end{obs}

We consider the square-free monomial ideal $I_{\alpha,\beta} = (x_{a_1}\cdots x_{b_1},x_{a_2}\cdots x_{b_2}, \ldots, x_{a_s}\cdots x_{b_s})$ in $S$. For $s=0$, we set $I_{\alpha,\beta}=0$. Our main result, is the following Theorem.

\begin{teor}
For any sequences of positive integers $\alpha$ and $\beta$ as above:

$(1)$ $\depth(S/I_{\alpha,\beta}) = \sdepth(S/I_{\alpha,\beta})= \varphi(\alpha,\beta)$.

$(2)$ $\dim(S/I_{\alpha,\beta})= \psi(\alpha,\beta)$.
\end{teor}

\begin{proof}
$(1)$ Denote $I=I_{\alpha,\beta}$. We use induction on $s\geq 0$ and $(a_s-a_1)+(b_s-b_1)\geq 0$. If $s=0$, then $I=0$ and there is nothing to prove. If $s=1$, then $I=(x_{a_1} \cdots x_{b_1})$ is a principal ideal. By Lemma $1.3(2)$, it follows that 
$\sdepth(S/I)=\depth(S/I)=n-1=\varphi(\alpha,\beta)$. Now, assume $s>1$. 

Denote $j=j(\alpha,\beta)$ and $I'=I_{\alpha',\beta'}$.
If $j=1$, then $u=x_{a_1}\cdots x_{b_1}$ is regular on $S/I'$ and $I=(u,I')$. By Lemma $1.3(2)$ and the induction hypothesis, it follows that $\sdepth(S/I)=\sdepth(S/I')-1=\varphi(\alpha',\beta')=\varphi(\alpha,\beta)$. Also, it is clear that
$\depth(S/I)=\depth(S/I')-1=\varphi(\alpha',\beta')=\varphi(\alpha,\beta)$.

Now, assume $s>1$ and $j>1$. Denote $I''=I_{\alpha'',\beta''}$. Let $u_0=1$ and $u_i=x_{b_1-i+1}\cdots x_{b_1}$ for all $1\leq i\leq b_1-a_1-1$. We consider the short exact sequences \[(\mathcal S_i) 0 \longrightarrow S/(I:u_i) \longrightarrow S/(I:u_{i-1}) \longrightarrow 
S/(x_{b_1-i+1},(I:u_{i-1})) \longrightarrow 0,(\forall)1\leq i\leq b_1-a_1-1. \]
Note that $S/(I:u_{b_1-i+1}) = (x_{a_1},I'')$ and therefore, by induction hypothesis and by Lemma $1.3(1)$ we have 
$\depth(S/I)\leq \depth(S/(I:u_1))\leq \cdots \leq \depth(S/(I:u_{b_1-i+1}))=\varphi(\alpha'',\beta'')-1$ and, similarly,
$\sdepth(S/I)\leq \sdepth(S/(I:u_1))\leq \cdots \leq \sdepth(S/(I:u_{b_1-i+1}))=\varphi(\alpha'',\beta'')-1$.
On the other hand, for all $1\leq i\leq b_1-a_1-1$, we have $(x_{b_1-i+1},(I:u_{i-1}))= (x_{b_1-i+1},I')$.

We claim that $\varphi(\alpha',\beta') \geq \varphi(\alpha'',\beta'')$. Assume the claim is true. From the short exact sequences 
$(\mathcal S_i)$, Lemma $1.1(2)$ and Lemma $1.2$, it follows that $\depth(S/(I:u_i))\geq \depth(S/(I:u_{i-1}))$ and 
$\depth(S/(I:u_i))\geq \depth(S/(I:u_{i-1}))$ for all $1\leq i\leq b_1-a_1-1$. In particular, $\depth(S/I)\geq \varphi(\alpha'',\beta'')-1$ and
$\sdepth(S/I)\geq \varphi(\alpha'',\beta'')-1$. Since the other inequality was already checked, it follows that $\sdepth(S/I)=\depth(S/I)=\varphi(\alpha,\beta)$, as required. In order to complete the proof, we must show the claim.

If $b_2<a_{j+1}$, then $I''=(x_{b_1+1}\cdots x_{b_2},I')$ and $x_{b_1+1}\cdots x_{b_2}$ is regular on $S/I'$. By Lemma $1.3(2)$, it follows that $\varphi(\alpha',\beta')=\varphi(\alpha'',\beta'')+1$. 

Suppose that $b_2\geq a_{j+1}$. If $I'=0$, i.e. $j=s$, there is nothing to prove. Assume $j<s$ and let $j'=j(\alpha',\beta')$. For $j'=1$, one can easily see that $\sdepth(S/I')=\sdepth(S/I'')$ and $\depth(S/I')=\depth(S/I'')$, thus our claim holds. 

If $j'>1$, let $u'=x_{a_{j+1}+1}\cdots x_{b_{j+1}}$. Then, by induction hypothesis on $I'$, it follows that $\depth(S/(I':u'))=\depth(S/I')$ and $\sdepth(S/(I':u'))=\sdepth(S/I')$. On the other hand, since $(I':u')=(I'':u'x_{b_1+1}\cdots x_{a_{j+1}-1})$, by Lemma $1.3(1)$, it follows that $\varphi(\alpha'',\beta'')\leq \varphi(\alpha',\beta')$, as required.

$(2)$ We use induction on $s\geq 0$. If $s\leq 1$, then there is nothing to prove. Assume $s\geq 2$. Denote $I=I_{\alpha,\beta}$ and $I'=I_{\alpha',\beta'}$. If $j(\alpha,\beta)=1$, then $u=x_{a_1}\cdots x_{b_1}$ is regular on $S/I'$ and $(I=u,I')$. By induction hypothesis, it follows that $\dim(S/I)=\dim(S/I')-1=\psi(\alpha',\beta')-1=\psi(\alpha,\beta)$. Now, assume $j>2$ and let $v=x_{a_2}\cdots x_{b_1}$. We have the short exact sequence $0 \rightarrow S/(I:v) \rightarrow S/I \rightarrow S/(I,v) \rightarrow 0$. Note that $(I:v)=(w,I'')$, where 
$w=x_{a_1}\cdots x_{a_2-1}$. Also, $(I,v)=(I',v)$. Since $\dim(S/I)=\max\{ \dim(S/(I:v)),\dim(S/(I,v)) \}$, in order to complete the proof, it is enough to show that $\dim(S/I')\geq \dim(S/I'')$.

Denote $u_k=x_{a_k}\cdots x_{b_k}$ for all $k\leq s$. Note that $I''=(x_{b_1+1}\cdots x_{b_2},I')$. If $a_{j+1}>b_1+1$ then $G(I'')=(x_{b_1+1}\cdots x_{b_2},u_{j+1},\cdots,u_{s})$. Else, $G(I'')=(x_{b_1+1}\cdots x_{b_2},u_{j+2},\cdots,u_{s})$. In either case, $j'=j(\alpha',\beta')\geq j''=j(\alpha'',\beta'')$. By induction hypothesis, $\dim(S/I')=\dim(S/I'_1)-1$ where $I'_1=(u_{j+j'+2},\ldots,u_s)$. Similarly, $\dim(S/I'')=\dim(S/I''_1)$, where $I''_1=(u_{j+j''+1},\ldots,u_{s})$, if $a_{j+1}>b_1+1$, or $I''_1=(u_{j+j''+2},\ldots,u_{s})$, if 
$a_{j+1}=b_1+1$. Note that $|G(I'_1)| \leq |G(I''_1)|$. If we apply the same procedure on $I'_1$ and $I''_1$ we will obtain new ideals $I'_2$
and $I''_2$ with  $|G(I'_2)| \leq |G(I''_2)|$. This procedure stops at the step $k$ when $I'_k=0$. It follows that $\dim(S/I')=n-d$. On the other hand, $\dim(S/I'')=\dim(S/I''_k)-k$. Thus, we are done.
\end{proof}

\begin{prop}
For any nonempty sequences of positive integers $\alpha$ and $\beta$ as above, we have:
$\sdepth(I_{\alpha,\beta})\geq \depth(I_{\alpha,\beta})=\varphi(\alpha,\beta)+1$.
\end{prop}

\begin{proof}
If $s=1$, then $I:=I_{\alpha,\beta}$ is a principal ideal, and therefore $\sdepth(I)=n=\varphi(\alpha,\beta)+1$. 
Assume $s\geq 2$.If $j:=j(\alpha,\beta)=1$, then $I=(x_{a_1}\cdots x_{b_1}, I') $, where $I'=I_{\alpha',\beta'}$. 
According to \cite[Corollary 2.5]{mirci} and the induction hypothesis, $$\sdepth(I) \geq \min\{ \sdepth(S/I'), \sdepth(I') \} = \varphi(\alpha',\beta') = \varphi(\alpha,\beta)+1.$$

Assume $j\geq 2$. Let $v=x_{a_2}\cdots x_{b_1}$. We can write $I=(I\cap (v))\oplus I/(I\cap(v))$. Note that $I\cap (v) = v(I:v)$, and therefore $\sdepth(I\cap (v)) = \sdepth(I:v) = \sdepth(x_{a_1}\cdots x_{a_2-1},I'')$, where $I''= I_{\alpha'',\beta''})$. By \cite[Corollary 2.5]{mirci} and the induction hypothesis, it follows that $\sdepth(I\cap (v))\geq \sdepth(S/I'') = \varphi(\alpha'',\beta'') = \varphi(\alpha,\beta)+1$. 

Let $S_1=K[x_1,\ldots,x_{b_1}]$ and $S_2=K[x_{b_1+1},\ldots,x_{n}]$. 
One can easily see that $I/(I\cap(v)) \cong S_1/vS_1 \otimes_K \bar I'$, where $\bar I'=I'\cap S_2$. By Lemma $1.4$, \cite[Lemma 3.6]{hvz} and the induction hypothesis, it follows that 
$\sdepth(I/(I\cap(v))) \geq \sdepth_{S_1}(S_1/vS_1)+\sdepth_{S_2}(\bar I') \geq (b_1-1)+\varphi(\alpha',\beta')-b_1+1 \geq \varphi(\alpha,\beta)+1$. Finally, by Lemma $1.2$, we get the required conclusion.
\end{proof}

\begin{obs}
\emph{According to \cite[Theorem 2.1]{okazaki}, $\sdepth(I_{\alpha,\beta})\geq n - \left\lfloor \frac{s}{2}\right\rfloor$.}
\end{obs}

\begin{prop}
Let $\alpha$ and $\beta$ as above. If $a_{k+2}>b_k+1$ for all $0\leq k\leq s-2$, then:

$(1)$ $\sdepth(S/I_{\alpha,\beta}) = \depth(S/I_{\alpha,\beta})=n-s$. 

$(2)$ $\sdepth(I_{\alpha,\beta})=n-\left\lfloor \frac{s}{2} \right\rfloor$, for all $s\geq 1$.

$(3)$ If $b_i\geq a_{i+1}$, for all $1\leq i\leq s-1$, then $\dim(S/I_{\alpha,\beta})=n-\left\lceil  \frac{s}{2} \right\rceil$.
\end{prop}

\begin{proof}
$(1)$ We use induction on $s\geq 0$. By Theorem $1.6$, it is enough to prove that $\varphi(\alpha,\beta)=n-s$. For $s=0$ there is nothing to prove. If $1\leq s\leq 2$, the conclusion follows from Remark $1.5$.

Assume $s\geq 3$. If $b_1>a_2$, i.e. $j(\alpha,\beta)=1$, note that $\alpha'$ and $\beta'$ satisfies the induction hypothesis, and therefore 
$\varphi(\alpha,\beta)=\varphi(\alpha',\beta')-1=n-(s-1)-1=n-s$. If $b_1\leq a_2$, then $j(\alpha,\beta)=2$, since $a_3>b_1$. Note that $\alpha''$ and $\beta''$ satisfy the induction hypothesis. Therefore, $\varphi(\alpha,\beta)=\varphi(\alpha'',\beta'')-1=n-(s-1)-1=n-s$.

$(2)$ Denote $I=I_{\alpha,\beta}$. Assume $s\geq 2$ and let $v= (x_{a_2}\cdots x_{b_1})(x_{a_3}\cdots x_{b_2})\cdots (x_{a_s}\cdots x_{b_{s-1}})$.  Note that $(I:v)=(x_{a_1}\cdots x_{a_2-1}, x_{b_1+1}\cdots x_{a_3-1}, \ldots, x_{b_{s-2}+1}\cdots x_{a_s-1} , x_{b_{s-1}+1}\cdots x_{b_s} )$ is a monomial complete intersection. Therefore, by \cite[Theorem 2.4]{shen}, $\sdepth(I:v)=n -\left\lfloor \frac{s}{2} \right\rfloor$. On the other hand, by Lemma $1.3(1)$ and Remark $1.4$, $\sdepth(I:v)\geq \sdepth(I) \geq n -\left\lfloor \frac{s}{2} \right\rfloor$. Thus, we are done.

$(3)$ We use induction on $s\geq 0$. If $s\leq 1$ then there is nothing to prove. If $s\geq 2$, then, by the induction hypothesis, we have $\dim(S/I)=\dim(S/I')-1 = n-\left\lceil  \frac{s-2}{2} \right\rceil - 1 =  n-\left\lceil  \frac{s}{2} \right\rceil$, as required.
\end{proof}

\begin{teor}
Let $\alpha$ and $\beta$ as above. If $a_{k+2}>b_k+1$ for all $0\leq k\leq s-2$, then $\sdepth(S/I_{\alpha,\beta}^t)= \depth(S/I_{\alpha,\beta}^t)=n-s$, for all $t\geq 1$.
\end{teor}

\begin{proof}
We use induction on $s\geq 0$ and $t\geq 1$. For $s\leq 1$, there is nothing to prove. The case $t=1$ was done in Proposition $1.9$. Assume $s\geq 2$ and $t\geq 2$. If $a_2> b_1$, then $u=x_{a_1}\cdots a_{b_1}$ is regular on $S/I_{\alpha',\beta'}$ and $I_{\alpha,\beta}=(u,I_{\alpha',\beta'})$. We consider the short exact sequence:
\[ 0 \longrightarrow S/(I_{\alpha,\beta}^t:u) \longrightarrow S/I_{\alpha,\beta}^t \longrightarrow S/(I_{\alpha,\beta}^t,u) \longrightarrow 0.\]
Note that $(I_{\alpha,\beta}^t:u) = I_{\alpha,\beta}^{t-1}$ and $(I_{\alpha,\beta}^t,u)=(I_{\alpha',\beta'}^t,u)$. By induction hypothesis and Lemma $1.3(2)$, we get $\sdepth(S/(I_{\alpha,\beta}^t:u))=\depth(S/(I_{\alpha,\beta}^t:u))=\sdepth(S/(I_{\alpha,\beta}^t,u))=\sdepth(S/(I_{\alpha,\beta}^t,u))=n-s$. From the above short exact sequence, Lemma $1.1(2)$, Lemma $1.2$ and Lemma $1.3(1)$, we get the required conclusion.

Now, assume $a_2\leq b_1$. Let $v=x_{a_2}\cdots x_{b_1}$, $w=x_{a_1}\cdots x_{a_2-1}$ and $u=vw$. We claim that 
$(I_{\alpha,\beta}^t:v^t) = (w, I_{\alpha'',\beta''})^t$. Since $I_{\alpha,\beta} = (vw,vx_{b_1+1}\cdots x_{b_2},I_{\alpha',\beta'})$, it follows that $(I_{\alpha,\beta}:v) = (w,I_{\alpha'',\beta''})$ and therefore $(I_{\alpha,\beta}^t:v^t) \supseteq (w, I_{\alpha'',\beta''})^t$. In order to prove the other inclusion, let $m\in S$ be a minimal monomial generator of $I_{\alpha,\beta}^t$. It follows that $m=u^i (vx_{b_1+1}\cdots x_{b_2})^j \cdot m'^k$, where $m'$ is a minimal monomial generator of $I_{\alpha',\beta'}$ and $i+j+k=t$. Let $\bar m\in S$ be a monomial such that $m|v^t \bar m$. It follows that $v^{i+j}m'^k| \bar m$, and therefore $\bar m \in (w, I_{\alpha'',\beta''})^t$, thus we proved our claim.
Note that $(w, I_{\alpha'',\beta''})=I_{\bar a,\bar b}$, where $\bar a: a_1<b_1+1<a_3<\cdots<a_s$ and $\bar b: a_2-1<b_2<\cdots<b_s$.
By Lemma $1.3(1)$ and the first part of the proof, it follows that $\depth(S/I_{\alpha,\beta}^t)\leq \depth(S/I_{\bar a,\bar b}^t)= n-s$ 
and $\sdepth(S/I_{\alpha,\beta}^t)\leq \sdepth(S/I_{\bar a,\bar b}^t)= n-s$. In order to prove the other inequalities, we consider short exact sequences:
\[0 \longrightarrow S/(I_{\alpha,\beta}^t:v^i) \longrightarrow S/(I_{\alpha,\beta}^t:v^{i-1}) \longrightarrow 
S/((I_{\alpha,\beta}^t:v^{i-1}),v) \longrightarrow 0,\;1\leq i\leq t.\]
Note that $((I_{\alpha,\beta}^t:v^{i-1}),v) = (I_{\alpha',\beta'}^t,v)$, for all $1\leq i\leq t$. On the other hand, by induction hypothesis and Lemma $1.3(2)$, $\sdepth(S/(I_{\alpha',\beta'}^t,v))=\depth(S/(I_{\alpha',\beta'}^t,v))=n-s+1$. If we apply repeatedly Lemma $1.1(2)$ and Lemma $1.2$ to the above exact sequences, we finally get $\depth(S/I_{\alpha,\beta}^t)\geq \depth(S/(I_{\alpha,\beta}^t:v^t)) = n-s$ and $\sdepth(S/I_{\alpha,\beta}^t)\geq \sdepth(S/(I_{\alpha,\beta}^t:v^t)) = n-s$. 
\end{proof}

\begin{prop}
Let $\alpha$ and $\beta$ as above. If $a_{k+2}=b_k+1$ for all $0\leq k\leq s-2$, then $\sdepth(S/I_{\alpha,\beta})= 
\depth(S/I_{\alpha,\beta})=n-s+\left\lfloor \frac{s}{3} \right\rfloor$.
\end{prop}

\begin{proof}
We use induction on $s\geq 0$. By Theorem $1.6$, it is enough to prove that $\varphi(\alpha,\beta)=n-s+\left\lfloor \frac{s}{3} \right\rfloor$. For $s=0$ there is nothing to prove. 
If $1\leq s\leq 2$, the conclusion follows from Remark $1.4$. If $s=3$, note that $I_{\alpha'',\beta''}=(x_{b_1+1},\ldots,x_{b_2})$ and therefore $\varphi(\alpha,\beta)=n-2$, as required.
Assume $s\geq 4$. Note that $I_{\alpha'',\beta''}=(x_{b_1+1},\ldots,x_{b_2},L)$, where $L=I_{\bar a,\bar b}$, $\bar a: a_4<a_5<\cdots<a_s$ ,
$\bar b: b_4<b_5<\cdots<b_s$. By induction hypothesis on $L$, it follows that $\varphi(\alpha,\beta)=\varphi(\bar a,\bar b)-2 = n-(s-3)+\left\lfloor \frac{s-3}{3} \right\rfloor - 2 = n-s+\left\lfloor \frac{s}{3} \right\rfloor$, as required.
\end{proof}

Let $\alpha$ and $\beta$ as above, and assume that $a_{k+2}=b_k+1$ for all $0\leq k\leq s-2$. For $1\leq \ell\leq s$,
we consider the sequences $\alpha|_{\ell}:a_1<a_2<\cdots<a_{\ell}$ and $\beta|_{\ell}:b_1<b_2<\cdots<b_{\ell}$.
We denote
$D(\ell,t)=\depth(S/I_{\alpha|_{\ell},\beta|_{\ell}}^t)$ and $S(\ell,t)=\sdepth(S/I_{\alpha|_{\ell},\beta|_{\ell}}^t)$. 

\begin{prop}
With the above notations, the following hold:

$(1)$ $D(s,1)=S(s,1)=n-s+\left\lfloor \frac{s}{3} \right\rfloor$, for any $s\geq 1$.

$(2)$ $D(s,t)=S(s,t)=n-s$, for any $t\geq 1$ and $s\leq 3$.

$(3)$ $D(s,t-1)\geq D(s,t)\geq \min\{D(s,t-1),D(s-1,t-1)-1,D(s-2,t)-1,D(s-3,t)-2\}$, for any $s\geq 3$ and $t\geq 2$.

$(4)$ $S(s,t-1)\geq S(s,t)\geq \min\{S(s,t-1),S(s-1,t-1)-1,S(s-2,t)-1,S(s-3,t)-2\}$, for any $s\geq 3$ and $t\geq 2$.
\end{prop}

\begin{proof}
$(1)$ It is a restatement of Proposition $1.11$. Assume $t\geq 2$. If $s=1$, then there is nothing to prove. Assume $s\geq 2$.

We denote $u_i=x_{a_i}\cdots x_{b_i}$, for all $1\leq i\leq s$. Also, we denote $I_i=(u_i,\cdots,u_s)$, for all $1\leq i\leq s$. For $i>s$, we set $I_i=0$. Let $v=x_{a_2}\cdots x_{b_1}$, $w=x_{a_1}\cdots x_{a_2-1}$, $w'=x_{b_1+1}\cdots x_{b_2}$ and $u=u_1$. Denote 
$I=I_1=I_{\alpha,\beta}$. We consider the short exact sequences:
\[ 0 \longrightarrow S/(I^t:v) \longrightarrow S/I^t \longrightarrow S/(I^t,v) \longrightarrow 0,\]
\[ 0 \longrightarrow S/(I^t:u) \longrightarrow S/(I^t:v) \longrightarrow S/((I^t:v),w) \longrightarrow 0, \]
\[ 0 \longrightarrow S/(I_2^t:u_2) \longrightarrow S/(I_2^t:v) \longrightarrow S/((I_2^t:v),w') \longrightarrow 0. \]
We claim that $(I^t:u)=I^{t-1}$. Indeed, {\small
$$(I^t:u)=(I_2^t+uI_2^{t-1}+\cdots+(u^t)):u = (I_2^t:u)+I_2^{t-1}+uI_2^{t-2}+\cdots+(u^{t-1})=(I_2^t:u)+I^{t-1}=I^{t-1},$$}
since $(I_2^t:u)\subset I_2^{t-1}$. Also, $(I^t,v)=(I_3^t,v)$ and $((I^t:v),w) = ((I_2^t:v),w)$. Moreover, $(I_2^t:u_2)=I_2^{t-1}$ and $((I_2^t:v),w')=(I_4^t,w')$.

Therefore, from the exact sequences above and Lemma $1.1(2)$, it follows that:
\[ \depth(S/I^t)\geq \min\{ \depth(S/(I^t:v)), \depth(S/I_3^{t})-1 \}, \]
\[ \depth(S/(I^t:v))\geq \min\{ \depth(S/I^{t-1}), \depth(S/(I_2^{t}:v))-1 \}, \]
\[ \depth(S/(I_2^{t}:v))\geq \min\{ \depth(S/I_2^{t-1}), \depth(S/I_4^{t})-1) \}. \]
By Lemma $1.2$, we get similar inequalities for $\sdepth$. From the above inequalities, it follows:
$$(i)\; \depth(S/I^t) \geq \min\{\depth(S/I^{t-1}),  \depth(S/I_2^{t-1})-1,\depth(S/I_3^{t})-1, \depth(S/I_4^{t})-2 \},$$ {\small
$$(ii)\; \sdepth(S/I^t) \geq \min\{\sdepth(S/I^{t-1}),  \sdepth(S/I_2^{t-1})-1,\sdepth(S/I_3^{t})-1, \sdepth(S/I_4^{t})-2 \}.$$}
On the other hand, since $(I^t:u)=I^{t-1}$, by Lemma $1.3(1)$, it follows that $\depth(S/I^t)\leq \depth(S/I^{t-1})$ and
$\sdepth(S/I^t)\leq \sdepth(S/I^{t-1})$. If $s=2$, then $I_2$ is principal and $I_3=I_4=(0)$. Therefore, using induction on $t\geq 1$, by $(i)$ and $(ii)$ it follows that $\sdepth(S/I^t)=\depth(S/I^t)=n-2$, for all $t\geq 1$.
\end{proof}

\begin{cor}
With the notations of Proposition $1.12$, for any $s,t\geq 1$, we have:

$(1)$ $n-s+ \left\lfloor \frac{s}{3} \right\rfloor \geq \depth(S/I_{\alpha,\beta}^t)\geq n-s +\max\{\left\lfloor \frac{s-t+1}{3} \right\rfloor,0\}.$

$(2)$ $n-s+ \left\lfloor \frac{s}{3} \right\rfloor \geq \sdepth(S/I_{\alpha,\beta}^t)\geq n-s +\max\{\left\lfloor \frac{s-t+1}{3}\right\rfloor,0\}.$
\end{cor}

\begin{proof}
We define, $d(s,t)=n-s$ for $0\leq s\leq 2$ and any $t\geq 1$. Also, we define $d(s,1)=n-s+\frac{s}{3}$, for any $s\geq 0$. We define $d(s,t)=\min\{d(s-1,t-1)-1,d(s-2,t)-1,d(s-3,t)-2\}$. In order to complete the proof, by Proposition $1.12$, it is enough to show that 
$d(s,t)=n-s +\max\{\left\lfloor \frac{s-t+1}{3} \right\rfloor,0\}$ for any $s,t\geq 1$. If $s\leq 2$ or $t = 1$, then we are done. Now, assume $s\geq 3$ and $t\geq 2$. By induction hypothesis, $d(s,t)=\min\{ n-(s-1)+ \max\{\left\lfloor \frac{(s-1)-(t-1)+1}{3} \right\rfloor,0\}-1, n-(s-2)-1+\max\{\left\lfloor \frac{(s-2)-t+1}{3} \right\rfloor, n-(s-3)+ \max\{\left\lfloor \frac{(s-3)-t+1}{3} \right\rfloor,0\}-2 \} = n-s +\max\{\left\lfloor \frac{s-t+1}{3}\right\rfloor,0\}$, as required.
\end{proof}

\begin{exm}
\emph{Let $n=10$, $\alpha: 1<2<3<6<7$ and $\beta: 4<5<7<8<10$. We have $I_{\alpha,\beta}=(x_1x_2x_3x_4, x_2x_3x_4x_5, x_3x_4x_5x_6x_7, x_6x_7x_8, x_7x_8x_9x_{10})$. Since $b_1=4>a_3=3$ and $b_1<a_4=6$, it follows that $j(\alpha,\beta)=3$. Therefore, $\alpha': 6<7$, $\beta': 8<10$, $\alpha'': 5<6<7$ and $\beta'': 5<8<10$. Note that $\varphi(\alpha'',\beta'')=\varphi(\alpha',\beta')-1 = (10-2)-1 = 7$. Thus, by Theorem $1.6$, it follows that $\sdepth(S/I_{\alpha,\beta})=\depth(S/I_{\alpha,\beta})=\varphi(\alpha,\beta)=6$. Also, by Remark $1.8$, $\sdepth(I_{\alpha,\beta}) \geq 10 - \left\lfloor \frac{5}{2} \right\rfloor = 8$. In fact, $\sdepth(I_{\alpha,\beta})=8$.}
\end{exm}

\subsection*{The path ideal of the path graph}

Let $n\geq m \geq 2$ be two integers. The \emph{path graph} of length $n$, denoted by $L_n$, is a graph with the vertex set $V=[n]$ and the edge set $E=\{\{1,2\},\{2,3\},\ldots,\{n-1,n\}\}$. We denote $I_{n,m} =(x_1x_2\cdots x_m,  x_2x_3\cdots x_{m+1}, \ldots, x_{n-m+1}x_{n-m+2}\cdots x_n )$. Note that $I_{n,m}$ is the $m$-path ideal of the graph $L_n$, provided with the direction given by $1<2<\ldots <n$, see \cite{tuy} for further details.

According to \cite[Theorem 1.2]{tuy}, 
$$pd(S/I_{n,m}) = \begin{cases} \frac{2(n-d)}{m+1},\; n\equiv d (mod\;(m+1))\;with\; 0 \leq d\leq m-1, \\ 
\frac{2n-m+1}{m+1},\; n\equiv m (mod\;(m+1)).
 \end{cases}$$ 
By Auslander-Buchsbaum formula (see \cite{real}), it follows that $\depth(S/I_{n,m})=n-pd(S/I_{n,m})$ and, by a straightforward computation, we can see $\depth(S/I_{n,m}) =  n+1 - \left\lfloor \frac{n+1}{m+1} \right\rfloor - \left\lceil \frac{n+1}{m+1} \right\rceil=:\varphi(n,m)$. 
We recall the following result from \cite{path}.

\begin{teor}(\cite[Theorem 1.3]{path})
 $\sdepth(S/I_{n,m})=\varphi(n,m)$.
\end{teor}

\begin{proof}
We use Theorem $1.6$. Note that $I_{n,m}=I_{\alpha,\beta}$, where $\alpha: 1<2<\cdots<n-m+1$ and $\beta: m<m+1<\cdots<n$. We use induction on $n\geq m$. If $n=m$, then $I_{n,m}=(x_1\cdots x_m)$ and there is nothing to prove. If $m<n\leq 2m$, then $I'':=I_{\alpha'',\beta''} = (x_{m+1})$, and thus $\sdepth(S/I_{n,m})=n-2=\varphi(n,m)$. Assume $n\geq 2m+1$. Then, $I''=(x_{m+1},x_{m+2}\cdots x_{2m+1},\ldots,x_{n-m+1}\cdots x_{n})$. Note that $S/I''\cong S/(I_{n-m-1,m}S,x_n)$. It follows that $\sdepth(S/I)=\sdepth(S/I'')-1=\varphi(n-m-1,m)+m-1 = \varphi(n,m)$, as required.
\end{proof}

\section{Hilbert series and Betti numbers}

In this section, we study the Hilbert series and the Betti numbers of the ideal $I_{\alpha,\beta}$. 
We consider two sequences of integers $\alpha: a_1 < a_2 <\cdots <a_s$ and $\beta: b_1 < b_2 <\cdots <b_s$ with 
$1\leq a_1$, $b_s\leq n$ and $a_i\leq b_i$, for all $1\leq i \leq s$. We consider the ideal
$$I:=I_{\alpha,\beta}=(x_{a_1}\cdots x_{b_1},\ldots,x_{a_s}\cdots x_{b_s})\subset S=K[x_1,\ldots,x_n].$$
As in the first section, we consider the sequences $\alpha',\beta',\alpha''$ and $\beta''$. We denote $I':=I_{\alpha',\beta'}$
and $I'':=I_{\alpha'',\beta''}$.

\begin{prop}
It holds that
$$H_{S/I}(t) = \begin{cases} (1+t+\cdots+t^{b_1-a_1})/(1-t)^{n-1},\;s=1 \\ (1-t^{b_1-a_1+1})H_{S/I'}(t),\;s>1,j=1 \\ 
t^{b_1-a_2+1}(1-t^{a_2-a_1})H_{S/I''}(t) + (1-t^{b_1-a_2+1})H_{S/I'}(t),\;s>1,j>1
 \end{cases}.$$
\end{prop}

\begin{proof}
If $s=1$, then $I=(x_{a_1})\cap(x_{a_2})\cap \cdots \cap(x_{b_1})$. Now, assume $s\geq 2$. Let $j=j(\alpha,\beta)$ and $I'=I_{\alpha',\beta'}$. For $j\geq 2$, we let $I''=I_{\alpha'',\beta''}$.

If $j=1$, then $I=(x_{a_1},I')\cap (x_{a_1+1},I')\cap \cdots \cap (x_{b_1},I')$. Assume $j\geq 2$.
Note that $I=(I:x_{b_1})\cap (I,x_{b_1})$ and $(I,x_{b_1})=(I',x_{b_1})$. Also $(I:x_{b_1})=(I:x_{b_1}x_{b_1-1})\cap ((I:x_{b_1}),x_{b_2})$. We repeat this procedure, until we get $(I:x_{b_1}\cdots x_{2})= ((I:x_{b_1}\cdots x_{2}),x_1)$. Note that, if $b_1-2\leq i\geq b_1-a_1$, then 
$((I:x_{b_1}\cdots x_{b_1-i}),x_{b_1-i-1}) = (I'',x_{b_1-i-1})$. Recursively, we obtain the irredundant primary decomposition of $I$.

Our next aim is to describe the Hilbert series of 
$S/I$ and its Betti numbers. 
Let $u=x_{a_1}\cdots x_{b_1}$. If $s=1$, then $I=(u)$ and thus $H_{S/I}(t)=(1+t+\cdots+t^{b_1-a_1})/(1-t)^{n-1}$. Assume $s\geq 1$.
If $j=1$, then $u=x_{a_1}\cdots x_{b_1}$ is regular on $S/I'$ and $I=(I',u)$. It follows that $H_{S/I}(t)=(1-t^{b_1-a_1+1})H_{S/I'}(t)$.
Assume $j>1$ and let $v=x_{a_2}\cdots x_{b_1}$. From the short exact sequence
$0 \longrightarrow S/(I:v) \longrightarrow S/I \longrightarrow S/(I,v) \longrightarrow 0$,
it follows that $H_{S/I}(t)=t^{b_1-a_2+1}H_{S/(I:v)}(t) + H_{S/(I,v)}(t)$.

Note that $S/(I:v) = S/(I'',w)$, where $w=u/v=x_{a_1}\cdots x_{a_2-1}$. Also, $w$ is regular on $S/I''$, and therefore 
$H_{S/(I:v)}=(1-t^{a_2-a_1})H_{S/I''}(t)$. Also, $(I,v)=(I',v)$ and $v$ is regular on $S/I'$. It follows that
$H_{S/(I,v)}(t)=(1-t^{b_1-a_2+1})H_{S/I'}(t)$. Thus, we get the required result.
\end{proof}

Since $\depth(I)=\varphi(\alpha,\beta)+1$, by Auslander-Buchsbaum Theorem, it follows that the projective dimension of $I$, is $p=pd(I)=n-\varphi(\alpha,\beta)-1$. We consider the minimal graded free resolution of $I$:
\[ 0 \rightarrow \bigoplus_{t\geq 0}S(-t)^{\beta_{pt}} \rightarrow \cdots \rightarrow \bigoplus_{t\geq 0}S(-t)^{\beta_{0t}} \rightarrow I \rightarrow 0,\]
where $\beta_{it}=\beta_{it}(I) = \dim_K Tor_i(I,K)_t$ are the \emph{graded Betti numbers} of $I$. The \emph{regularity} of $I$ is $\reg(I)=\max\{t-i:\;\beta_{it}(I)\neq 0\}$. Note that, if $s=1$, then $\beta_{0d}(I)=b_1-a_1+1$ is the only nonzero Betti number. In the following, we will assume $s\geq 2$.

We recall a definition introduced in \cite{eli}. A monomial ideal $I\subset S$ is \emph{splittable}, if  
$I$ is the sum of two non-zero monomial ideals $J$ and $L$ such that $G(I)=G(J)\cup G(L)$ and there is a splitting function
$G(J\cap L) \rightarrow G(J)\times G(L)$, $w \mapsto (\phi(w),\psi(w))$, such that $w=\lcm(\phi(w),\psi(w))$ and for every nonempty
subset $G'\subset G(J\cap L)$, $\lcm(\phi(G'))$ and $\lcm(\psi(G'))$ strictly divide $\lcm(G')$. Under these assumptions, the following holds:

\begin{prop}(Eliahou-Kervaire \cite{eli} and Fattabi \cite{fatabi})
For all $i,t>0$, we have $\beta_{it}(I) = \beta_{it}(J)+\beta_{it}(L)+\beta_{i-1,t}(J\cap L)$.
\end{prop}

Denote $u_i=x_{a_i}\cdots x_{b_i}$, for all $1\leq i\leq s$. Let $\bar \alpha: a_2<a_3<\cdots <a_s$, $\bar \beta: b_2<b_3<\cdots <b_s$ and 
$\bar I=I_{\bar \alpha,\bar \beta}$. Note that $I=(u,\bar I)$, where $u=u_1$, and $G(\bar I)=\{u_2,\ldots,u_s\}$. Also, 
$(u)\cap \bar I = u (\bar I:u)$. If $j(\alpha,\beta)=1$ then $(\bar I : u) = \bar I = I'=I_{\alpha',\beta'}$. Else, $(\bar I : u) = I''$. In both cases, note that $I=(u)+\bar I$ is a splitting in the sense of the above definition. Indeed, for $j=1$, define $\phi(u\cdot u_k)=u$ and $\psi(u\cdot u_k)$ for all $2\leq k\leq s$. In the second case, since $uI'' = (\bar u = x_{a_1}\cdots x_{b_2}, uu_{j+1},\cdots uu_{s})$, we can define $\phi(\bar u)=u$, $\psi(\bar u)=x_{b_1+1}\cdots x_{b_2}$, $\phi(uu_k)=u$ and $\psi(uu_k)=u_k$, for all $k\geq j+1$. Thus, as a direct consequence of Proposition $2.2$, we get:

\begin{prop}
With the above notations, we have:

(1) If $j=1$, then $\beta_{it}(I)=\beta_{it}((u))+\beta_{it}(I')+\beta_{i-1,t-\deg(u)}(I')$.

(2) If $j>1$, then $\beta_{it}(I)=\beta_{it}((u))+\beta_{it}(\bar I)+\beta_{i-1,t-\deg(u)}(I'')$.
\end{prop}

\begin{exm}
\emph{Let $n=8$, $\alpha: 1<2<4<6$ and $\beta: 3<5<7<8$. We have $I:=I_{\alpha,\beta}=(x_1x_2x_3, x_2x_3x_4x_5, x_4x_5x_6x_7, x_6x_7x_8)$.
Note that $j=j(\alpha,\beta)=2$. It follows that $\bar I = (x_2x_3x_4x_5, x_4x_5x_6x_7, x_6x_7x_8)$, $I'=I_{\alpha',\beta'}=(x_4x_5x_6x_7, x_6x_7x_8)$ and $I''=I_{\alpha'',\beta''}=(x_4x5,x_6x_7x_8)$.}

\emph{According to Proposition $2.1$, we get $H_{S/I}(t)=t^2(1-t)H_{S/I''}(t) + (1-t)H_{S/I'}(t)$. On the other hand, $H_{S/I''}(t)=
\frac{(1-t^2)(1-t^3)}{(1-t)^8}$. Also, $H_{S/I'}(t)=\frac{t^2(1-t^2)}{(1-t)^8} + \frac{(1-t^2)}{(1-t)^8}$. It follows that:
\[H_{S/I}(t)=\frac{t^2(1+t)(1+t+t^2)}{(1-t)^5} + \frac{(t^2+1)(1+t)}{(1-t)^6} = \frac{(1+t)(-t^5+2t^2+1)}{(1-t)^6}. \]
Note that $\dim(S/I)=6$ and $\sdepth(S/I)=\depth(S/I)=5$. Also, $\sdepth(I)=7$. We leave as an exercise to the reader, the primary decomposition of $I$ and the Betti numbers.}
\end{exm}

\vspace{2mm} \noindent {\footnotesize
\begin{minipage}[b]{15cm}
Mircea Cimpoea\c s, Simion Stoilow Institute of Mathematics, Research unit 5, P.O.Box 1-764,\\
Bucharest 014700, Romania, E-mail: mircea.cimpoeas@imar.ro
\end{minipage}}

\end{document}